\documentclass[a4paper,11pt]{amsart}

\usepackage[left=2.7cm, right=2.7cm]{geometry}

\usepackage{amsmath}
\usepackage{amssymb} 
\usepackage{amsfonts}
\usepackage{enumerate}
\usepackage{graphicx}

\theoremstyle{plain}
\newtheorem{theorem}{Theorem}[section]
\newtheorem{proposition}[theorem]{Proposition}
\newtheorem{lemma}[theorem]{Lemma}
\newtheorem{corollary}[theorem]{Corollary}

\theoremstyle{definition}
\newtheorem{definition}[theorem]{Definition}
\newtheorem{example}[theorem]{Example}
\newtheorem{remark}[theorem]{Remark}

\numberwithin{equation}{section}
\begin{document}

\title{The structure of multigranular rough sets}

\author{Jouni J{\"a}rvinen}
\address[J.~J{\"a}rvinen]{Department of Mathematics and Statistics, University of Turku, 20014 Turku, Finland}
\email{jjarvine@utu.fi}

\author{S{\'a}ndor Radeleczki}
\address[S.~Radeleczki]{Institute of Mathematics, University of Miskolc, 3515~Miskolc-Egyetemv{\'a}ros, Hungary}
\email{matradi@uni-miskolc.hu} 
           
\thanks{The work of the second author was carried out as a part of the EFOP-3.6.1-16-00011 ``Younger and
Renewing University -- Innovative Knowledge City'' project implemented in the framework of the
Sz{\'e}chenyi 2020 program, supported by the European Union, co-financed by the European Social
Fund.}

\begin{abstract}
We study multigranulation spaces of two equivalences.
The lattice-theoretical properties  of so-called ``optimistic'' and ``pessimistic''  multigranular approximation systems
are given. We also consider the ordered sets of rough sets determined by these approximation pairs.
\end{abstract}

\keywords{Equivalence relation, multigranular approximation, definable set, rough set, coherence, tolerance relation,
  irredundant covering, atomistic Boolean lattice, completely distributive lattice, Dedekind--MacNeille completion}

\maketitle

%
%
%
%
%
%
%
%

\section{Introduction and preliminaries} \label{Sec:IntroPrel}

Rough approximation operations were introduced by Z.~Pawlak in \cite{Pawl82}. In rough set theory, it is assumed
that our knowledge about the objects of a universe of discourse $U$ is given in terms of an equivalence
relation $E$ on $U$. In the literature, there are generalizations in which rough sets are defined in terms
of an arbitrary binary relation \cite{YaoLin1996}.

A \emph{tolerance} is reflexive and symmetric binary relation. In this work, tolerances are interpreted as similarity relations.
Let $T$ be a tolerance on a set $U$. If $x \, T \, y$, then $x$ and $y$ are considered similar
in terms of the information represented by $T$. If a tolerance $E$ is also transitive, it is an
equivalence relation and $E$ is interpreted as an \emph{indistinguishablity} relation such that $x \, E \, y$ means that 
we are not able to distinguish $x$ from $y$ in terms of the information given by $E$.
The properties of rough approximations defined by tolerances are well-studied; see e.g.~\cite{JarRad19} for more details.

For any $x \in U$, we denote by
\[ T(x) = \{ y \in U \mid x \, T \, y \} \]
the \emph{$T$-neighbourhood} of $x$. It consists of objects that are $T$-similar to $x$.
For any $X \subseteq U$, the \emph{lower $T$-approximation} is
\[
X_T = \{ x \in U \mid T(x) \subseteq X \}.
\]
The set $X_T$ is considered as the set of objects certainly belonging to $X$, because if $x \in X_T$,
then all objects $T$-similar to $x$ are in $X$. The \emph{upper approximation} of $X$ is
\[
X^T = \{ x \in U \mid T(x) \cap X \neq \emptyset \}.
\]
The set $X^T$ can be viewed as the set of elements possibly belonging to $X$, because  $x \in X^T$
means that in $X$ there is at least one element $T$-similar to $x$. For $X \subseteq U$, we denote
the \emph{complement} $U \setminus X$ by $X^c$ and $\wp(U) = \{X \mid X \subseteq U\}$ is the \emph{power set} of $U$.
The following properties are well known for all $X,Y \subseteq U$ and $\mathcal{H} \subseteq \wp(U)$:
\begin{enumerate}[\rm ({T}1)]
\item $\emptyset_T = \emptyset^T = \emptyset$ and $U_T = U^T = U$;
\item $(\bigcup \mathcal{H})^T = \bigcup \{X^T \mid X \in \mathcal{H} \}$ and
  $(\bigcap \mathcal{H})_T = \bigcap \{X_T \mid X \in \mathcal{H} \}$:
\item $(X^c)^T = (X_T)^c$ and  $(X^c)_T = (X^T)^c$;
\item $X_T \subseteq X \subseteq X^T$;
\item $X \subseteq Y$ implies $X_T \subseteq Y_T$ and $X^T \subseteq Y^T$;
\item $(X_T)^T \subseteq X \subseteq (X^T)_T$.
\end{enumerate}
Note that all these conditions are not independent and from these it follows that $X^T = ((X^T)_T)^T$ and $X_T = ((X_T)^T)_T$.
In addition, if $E$ is an equivalence, then
\begin{enumerate}[\rm ({E}1)]
\item $(X_E)^E = (X_E)_E = X_E$ \ and \ $(X^E)_E = (X^E)^E = X^E$.
\end{enumerate}
It is now clear that for an equivalence $E$ on $U$, the map $X \mapsto X^E$ is a lattice-theoretical \emph{closure operator}, that is,
it is extensive, order-preserving and idempotent; see e.g.\@\cite{bs81} for the definition. Note that the notion of a lattice-theoretical
closure operator  is more general than that of a topological closure operator since we do not require that the union of two closed
subsets be closed. 
We also have that $X \mapsto X_E$ is a lattice-theoretical \emph{interior operator}, that is, it is contractive, order-preserving,
and idempotent.

Note that in this work, we denote an ordered set $(L,\leq)$ often simply by $L$ if there is no danger of confusion about the order. 
Let $L$ be a lattice with a least element $0$. An element $a \in L$ is an \emph{atom} if $0 \prec a$, that is, $0$ is covered by $a$.
The lattice $L$ is atomistic, if each its element $x$ is the join of the atoms below $x$.

Let us denote 
\[
\wp(U)^E = \{ X^E \mid X \subseteq U \} \text{ \ and \ } \wp(U)_E = \{ X_E \mid X \subseteq U \}. \]
We say that a set $X \subseteq U$ is \emph{saturated} by an equivalence $E$ if $X$ is a (possibly empty) union of $E$-classes.
Let us denote by $\mathrm{Sat}(E)$ the family of all $E$-saturated sets. It is well known that 
\[ \mathrm{Sat}(E) = \wp(U)_E = \wp(U)^E.\] 
The ordered set $(\mathrm{Sat}(E),\subseteq)$ is a complete Boolean lattice in which
\begin{equation}\tag{E2}
\bigvee \mathcal{H} = \bigcup \mathcal{H} \quad \text{ \ and \ } \quad \bigwedge \mathcal{H} = \bigcap \mathcal{H} 
\end{equation}
for any $\mathcal{H} \subseteq \mathrm{Sat}(E)$. Additionally, $\emptyset$ is the least element and $U$ is the greatest element of 
$\mathrm{Sat}(E)$. The Boolean complement of $X \in \mathrm{Sat}(E)$ is the set-theoretical complement $X^c$. In addition,
$\mathrm{Sat}(E)$ is an atomistic lattice in which the atoms are the $E$-classes. 

For a tolerance $T$, the map $X \mapsto X^T$ is not a closure operator, because $(X^T)^T = X^T$ does not
necessarily hold. However, the operator $\Diamond \colon \wp(U) \to \wp(U)$ defined by
\[ \Diamond X = ( X^T)_T  \]
is a closure operator on $U$, and the corresponding \emph{closure system} is
\[
\wp(U)_T = \{ X_T \mid X \subseteq U \} =  \{ \Diamond X \mid X \subseteq U \} =  \{ A \subseteq U \mid \Diamond A = A \}.
\]
From the general properties of closure operators it follows that $\wp(U)_T$ is a complete lattice in which 
\begin{equation}\tag{T6}
\bigwedge \mathcal{H} = \bigcap \mathcal{H} \qquad \text{and} \qquad \bigvee \mathcal{H} = \Diamond \big( \bigcup \mathcal{H} \big)
\end{equation}
for $\mathcal{H} \subseteq \wp(U)_T$. Similarly,
\[ \Box X =  ( X_T)^T \]
defines an interior operator on $U$, and the corresponding interior system is
\[
\wp(U)^T = \{ X^T \mid X \subseteq U \} =  \{ \Box X \mid X \subseteq U \} =  \{ A \subseteq U \mid \Box A = A \}.
\]
Hence, $\wp(U)^T$ is a complete lattice in which 
\begin{equation}\tag{T7}
\bigvee \mathcal{H} = \bigcup \mathcal{H} \qquad \text{and} \qquad \bigwedge \mathcal{H} = \Box \big( \bigcap \mathcal{H} \big)
\end{equation}
for any $\mathcal{H} \subseteq \wp(U)^T$. 

A mapping $x \mapsto x^\bot$ on  a bounded lattice $L$ is called an \emph{orthocomplementation},
and $x^\bot$ an \emph{orthocomplement} of $x$, if for all $x, y \in L$: 
\begin{enumerate}[({O}1)]
 \item $x \leq y$ implies $y^\bot \leq x^\bot$ \hfill (order-reversing)
 \item $x^{\bot \bot} = x$ \hfill (involution)
 \item $x \vee x^\bot = 1$ and $x \wedge x^\bot = 0$  \hfill (complement)
\end{enumerate}
An \emph{ortholattice} is a bounded lattice equipped with an orthocomplementation. Note that orthocomplementations are not always unique.
We noted in \cite{JarRad14} that for any tolerance $T$ on $U$, $\wp(U)^T$ and $\wp(U)_T$ are ortholattices. 
For any $A \in \wp(U)^T$, its orthocomplement is $(A^c)^T$. Analogously, the orthocomplement of $B \in \wp(U)_T$ is
$(B^c)_T$.

A complete lattice $L$ is \emph{completely distributive} if for any doubly indexed
subset $\{x_{i,\,j}\}_{i \in I, \, j \in J}$ of $L$, 
\[
\bigwedge_{i \in I} \Big ( \bigvee_{j \in J} x_{i,\,j} \Big ) = 
\bigvee_{ f \colon I \to J} \Big ( \bigwedge_{i \in I} x_{i, \, f(i) } \Big ), \]
that is, any meet of joins may be converted into the join of all possible elements obtained by taking the meet over $i \in I$ of
elements $x_{i,\,k}$\/, where $k$ depends on $i$.  A collection $\mathcal{H}$ of nonempty subsets of $U$ is called a
\emph{covering} of $U$ if $\bigcup \mathcal{H} = U$. A covering $\mathcal{H}$ is \emph{irredundant} if
$\mathcal{H} \setminus \{X\}$ is not a covering for any $X \in \mathcal{H}$.
Each covering $\mathcal{H}$ defines a tolerance $T_\mathcal{H} = \bigcup \{ X^2 \mid X \in \mathcal{H}\}$, 
called the \emph{tolerance induced} by $\mathcal{H}$. 

We proved in \cite{JarRad14} that for a tolerance $T$ on $U$, the complete lattices $\wp(U)_T$ and $\wp(U)^T$ are completely 
distributive if and only if $T$ is induced by an irredundant covering of $U$. From this it follows that  $\wp(U)_T$ and $\wp(U)^T$
are Boolean lattices when $T$ is induced by an irredundant covering.

For a tolerance $T$ on $U$, a nonempty subset $X$ of $U$ is a \emph{$T$-preblock} if $X \times X \subseteq T$. 
A \emph{$T$-block} is a $T$-preblock that is maximal with respect to the inclusion relation. 
Each tolerance $T$ is induced by its blocks, that is, $a \, T \, b$  if and only if there exists a block 
$B$ such that $a,b \in B$. Note that the covering consisting of $T$-blocks is not necessarily irredundant.

In \cite{JarRad19}, we showed that for all $x \in U$, $T(x)$ is a $T$-block if and only it is a $T$-preblock. We also proved that
\begin{enumerate}[({T}1)]
\setcounter{enumi}{7}
\item $T$ is induced by an irredundant covering if and only if  $\{ T(x) \mid \text{$T(x)$ is a block}\}$ induces $T$.
\end{enumerate}
A complete Boolean lattice is atomistic if and only if it is completely distributive (see e.g.\@ \cite{Grat98}). 
Thus, if $T$ is induced by an irredundant covering of $U$, then the complete lattices
$\wp(U)^T$ and  $\wp(U)_T$ are atomistic Boolean lattices. We showed in \cite{JarRad14} that $\{ T(x) \mid T(x) \text{ is a block}\, \}$ 
and $\{ T(x)_T \mid T(x) \text{ is a block}\, \}$ are their sets of atoms, respectively.

Let $T$ be a tolerance on $U$. A set $X \subseteq U$ is called \emph{$T$-definable} if $X_T = X^T$.
This means that the set of elements which certainly are in $X$ coincides with the set of elements which
possibly are in $X$. We denote by $\mathrm{Def}(T)$ the family of $T$-definable sets.
It is a well-known fact (see e.g. \cite{Jarv99}) that for all $X \subseteq U$,
\[ X \in \mathrm{Def}(T) \iff X = X^T \iff X = X_T \iff X \in \mathrm{Sat}(T^e), \]
where $T^e$ is the smallest equivalence containing $T$. Therefore, $(\mathrm{Def}(T),\subseteq)$ is a complete
atomistic Boolean lattice. For an equivalence $E$, we have 
\[ \mathrm{Def}(E) = \mathrm{Sat}(E) = \wp(U)_E = \wp(U)^E.\]

Let $T$ be a tolerance on $U$. The rough \emph{$T$-equality} is a binary relation defined on $\wp(U)$ by
\[
X \equiv_T Y \iff X_T = Y_T \text{ \ and \ } X^T = Y^T .\]
This means that $X$ and $Y$ are roughly $T$-equal if the same elements belong possibly and certainly to $X$ and $Y$ in 
view of $T$. The relation $\equiv_T$ is an equivalence on $\wp(U)$ and its equivalence classes are called $T$-\emph{rough sets}, 
or simply \emph{rough sets}. 
Each rough set $\mathcal{C} \in \wp(U) / {\equiv_T}$ is determined by the approximation pair $(X_T, X^T)$, where $X$ is any
member of $\mathcal{C}$, as was originally pointed out by T.~B.~Iwi{\'n}ski \cite{Iwin87}.
Therefore, the following collection can be viewed as the set of all $T$-rough sets on $U$:
\[\mathit{RS}(T) = \{(X_T, X^T ) \mid X \subseteq U \}.\]
The set $\mathit{RS}(T)$ can be ordered coordinatewise by 
\[
(X_T, X^T ) \leq (Y_T, Y^T ) \iff X_T \subseteq Y_T \text{ \ and \ } X^T \subseteq Y^T. 
\]
The set $\mathit{RS}(T)$ is bounded with $(\emptyset,\emptyset)$ and $(U,U)$ as the least and the greatest element, respectively.

An element $x^*$ is the \emph{pseudocomplement} of $x$ if $x \wedge x^* = 0$ and $x \wedge z = 0$ implies $z \leq x^*$.  
A lattice $L$ in which each element has a pseudocomplement is called a \emph{pseudocomplemented lattice}.
A distributive pseudocomplemented lattice is a \emph{Stone lattice} if it satisfies the identity
\begin{equation} \label{Eq:Stone} \tag{St1}
x^* \vee x^{**} = 1.
\end{equation}

Similarly, an element $x^+$ is the \emph{dual pseudocomplement} of $x$ whenever $z \geq x^+$ is equivalent to $x \vee z = 1$.
If $L$ is such that each of its elements have a pseudocomplement and a dual pseudocomplement, then $L$ is a 
\emph{double pseudocomplemented lattice}. A double pseudocomplement lattice is \emph{regular} if it satisfies the identity
\begin{equation}
x^* = y^* \text{ and } x^+ = y^+ \text{ imply } x=y. \tag{M}
\end{equation}

A \emph{double Stone lattice} is Stone lattice in which every element has a dual pseudocomplement satisfying
\begin{equation} \label{Eq:dualStone} \tag{St2}
x^+ \wedge x^{++} = 0.
\end{equation}
It was proved by J.~Pomyka{\l}a and J.~A.~Pomyka{\l}a \cite{PomPom88} that for any equivalence $E$ on $U$, $\mathit{RS}(E)$ forms
a complete lattice such that
\begin{align}\label{Eq:RS(E)_meet}
\bigwedge_{X \in \mathcal{H}} (X_E,X^E) &= \Big ( \bigcap_{X \in \mathcal{H}} X_E,  \bigcap_{X \in \mathcal{H}} X^E \Big ) \\
\intertext{and}
\label{Eq:RS(E)_join}
\bigvee_{X \in \mathcal{H}} (X_E,X^E) &= \Big ( \bigcup_{X \in \mathcal{H}} X_E, \bigcup_{X \in \mathcal{H}} X^E \Big  ) .
\end{align}
In addition, they proved that each element $(A,B)$ of $\mathit{RS}(E)$ has a pseudocomplement
\[ (A,B)^* = (B^c,B^c) \]
and that $\mathit{RS}(E)$ forms a Stone lattice.
The result was complemented by S.~D.~Comer in \cite{Comer} by showing that $\mathit{RS}(E)$ forms actually a regular double Stone lattice, 
in which 
\[
(A, B)^+ = (A^c, A^c )
\]
for $(A,B) \in \mathit{RS}(E)$.

It is known that if $T$ is a tolerance, then $\mathit{RS}(T)$ is not necessarily even a lattice \cite{Jarv99}. 
We showed in \cite{JarRad14} that $\mathit{RS}(T)$ is a complete lattice if and only if $\mathit{RS}(T)$ is a complete sublattice of 
the direct product $\wp(U)_T \times \wp(U)^T$.  This means that if $\mathit{RS}(T)$ is a complete lattice, then for 
$\{(X_T,X^T)\}_{X \in \mathcal{H}} \subseteq \mathit{RS}(T)$,
\begin{align}\label{Eq:RS_lattice_meet}
\bigwedge_{X \in \mathcal{H}} (X_T,X^T) &= \Big ( \bigcap_{X \in \mathcal{H}} X_T, 
\Box \big ( \bigcap_{X \in \mathcal{H}} X^T \big ) \Big ) \\
\intertext{and}
\label{Eq:RS_lattice_join}
\bigvee_{X \in \mathcal{H}} (X_T,X^T) &= \Big ( \Diamond \big ( \bigcup_{X \in \mathcal{H}} X_T \big ), 
         \bigcup_{X \in \mathcal{H}} X^T \Big  ) .
\end{align}
In addition, we proved \cite[Theorem~4.8]{JarRad14} that $\mathit{RS}(T)$ is a completely distributive lattice if and only if 
$T$ is induced by an irredundant covering. We also showed in \cite{JarRad18} that if $T$ is a tolerance induced by an irredundant covering, 
then $\mathit{RS}(T)$ forms a regular double pseudocomplemented lattice such that for any $(A,B) \in \mathit{RS}(T)$,
\[
(A,B)^* = ((B^c)_T, (B^c)^T)  \text{\qquad and \qquad} (A,B)^+ = ((A^c)_T, (A^c)^T) .
\]

 For an ordered set $(P,\leq)$, a mapping ${\sim} \colon P \to P$ satisfying (O1) and (O2) is called a \emph{polarity}.
Such a polarity $\sim$ is an order-isomorphism from $(P,\leq)$ to its dual $(P,\geq)$, and we say that $P$ is \emph{self-dual}.
The Hasse diagram of a self-dual ordered set looks the same when it is  turned upside-down. An ordered set may have several
polarities. If $L$ is a complete lattice with a polarity $\sim$, then for all $S \subseteq L$,
\[ {\sim} \big ( \bigvee S \big ) = \bigwedge \{ {\sim} x \mid x \in S\} \quad \text{and} \quad
   {\sim} \big ( \bigwedge S \big ) = \bigvee \{ {\sim} x \mid x \in S\}. \]
A complete lattice with a polarity is called a \emph{complete polarity lattice}.   
Note also that a pseudocomplemented lattice $L$ with a polarity $\sim$ is a double pseudocomplemented lattice in which
for $x \in L$,
\[
{\sim}(x^+) = ({\sim} x)^* \qquad \text{and} \qquad {\sim}(x^*) = ({\sim} x)^+.
\]
For any tolerance $T$, $\textit{RS}(T)$ has a polarity defined for $(A,B) \in \textit{RS}(T)$ by 
\[ {\sim} (A,B) = (B^c, A^c).\]

A \emph{De~Morgan algebra}   is a structure $(L, \vee, \wedge, {\sim}, 0, 1)$ such that
$(L, \vee, \wedge, 0, 1)$ is a bounded distributive lattice equipped with a polarity $\sim$. If a De~Morgan algebra satisfies 
the inequality 
\begin{equation} \tag{K}
x \wedge {\sim} x \leq y \vee {\sim} y, 
\end{equation}
it is called a \emph{Kleene algebra}. 
We noted in \cite{JarRad14} that if $T$ is a tolerance induced by an irredundant covering of $U$, 
then $\textit{RS}(T)$ forms a Kleene algebra.

In the literature  \cite{Qian2010,She2012} can be found studies in which approximation operators are defined by using certain combinations
of two equivalence relations, meaning that concepts are described by two ``granulation spaces''.
This work is devoted to the study of lattice-theoretical structures arising from such multigranulation spaces.
In Section~\ref{Sec:multi} we study the basic lattice-theoretical properties of so-called ``optimistic'' and ``pessimistic''
approximations.
The ordered sets of rough sets determined by these approximation pairs are considered in Section~\ref{Sec:RoughSets}.
Particularly, the rough set system determined by optimistic approximations
does not necessarily form a lattice and its Dedekind--MacNeille completion is considered in Section~\ref{Sec:Completion}.
Some concluding remarks end the work.

\section{Multigranular approximations} \label{Sec:multi}

Let $P$ and $Q$ be equivalences on a set $U$. For any $X \subseteq U$, the so-called \emph{optimistic lower approximation} of $X$ 
\cite{Qian2010} is defined as
\[ X_{P + Q} = \{ x \in U \mid P(x) \subseteq X \text{ \ or \ } Q(x) \subseteq X \}. \] 
For instance, if $P$ is interpreted as a knowledge of one expert and $Q$ represents knowledge of a second one, then
$X_{P + Q}$ can be seen as a lower approximation of $X$ such that an element belongs to $X_{P + Q}$ if and only if  it is certainly in $X$ 
by the knowledge of at least one of the experts. The \emph{optimistic upper approximation} of $X$ is defined as the dual of
$X \mapsto X_{P + Q}$ by setting 
\[ X^{P + Q} = \big ((X^c)_{P + Q} \big )^c . \]

It can be easily seen that 
\[ X^{P + Q} = \{ x \in U \mid P(x) \cap X \neq \emptyset \text{ \ and \ } Q(x) \cap X \neq \emptyset \}. \]
The operators can be also written in the form
\[ X_{P + Q} = X_{P} \cup X_{Q} \text{ \quad and \quad } X^{P + Q} = X^{P} \cap X^{Q}. \]

The following properties are given in \cite{Qian2010}: 
\[ 
X \subseteq  X^{P + Q}, \qquad \left (X^{P + Q} \right )^{P + Q} =  X^{P + Q}, \qquad X \subseteq Y \Longrightarrow X^{P + Q} \subseteq Y^{P + Q}.
\]
This means that $X \mapsto X^{P + Q}$ is a closure operator. The corresponding closure system is
\[ \wp(U)^{P + Q} = \{ X^{P + Q} \mid X \subseteq U \} =  \{A \subseteq U \mid A^{P + Q} = A\}.\]
The ordered set $(\wp(U)^{P + Q}, \subseteq)$ is a complete lattice in which
\[ 
\bigwedge \mathcal{H} = \bigcap \mathcal{H} \qquad \text{and} \qquad \bigvee \mathcal{H} =  \big(\bigcup \mathcal{H}\big)^{P + Q}.
\]
Similarly, the map $X \mapsto X_{P + Q}$ is an \emph{interior operator}, that is, 
\[ 
X_{P + Q} \subseteq  X, \qquad \left (X_{P + Q} \right )_{P + Q} =  X_{P + Q}, \qquad X \subseteq Y \Longrightarrow X_{P + Q} \subseteq Y_{P + Q}.
\]
The family 
\[ \wp(U)_{P + Q} = \{ X_{P + Q} \mid X \subseteq U \} =  \{A \subseteq U \mid A_{P + Q} = A\} \]
is an \emph{interior system}.  The ordered set $(\wp(U)_{P + Q}, \subseteq)$ is a complete lattice such that
\[ 
\bigvee \mathcal{H} = \bigcup \mathcal{H} \qquad \text{and} \qquad \bigwedge \mathcal{H} =  \big(\bigcap \mathcal{H}\big)_{P + Q}.
\]

\begin{lemma} The complete lattices $\wp(U)^{P + Q}$ and $\wp(U)_{P + Q}$ are dually isomorphic, that is, 
\[ (\wp(U)^{P + Q},\subseteq) \cong (\wp(U)_{P + Q},\supseteq). \]
\end{lemma}

\begin{proof} We show that the map $\varphi \colon  X \mapsto X^c$ is an order-reversing isomorphism from $\wp(U)^{P + Q}$ to $\wp(U)_{P + Q}$.
First we note that $\varphi$ is well-defined. If $X \in \wp(U)^{P + Q}$, then
\[ \varphi(X) = \varphi( X^{P + Q} ) = ( X^{P + Q} )^c = (X^c)_{P + Q}, \]
which clearly belongs to $\wp(U)_{P + Q}$. Let $X,Y \in  \wp(U)^{P + Q}$. 
Now $X \subseteq Y$ is equivalent to $X^c \supseteq Y^c$, which means that $X \subseteq Y$ if and only if
$\varphi(X) \supseteq \varphi(Y)$. Therefore, $\varphi$ is an order-embedding. Note that an order-embedding
is always an injection. Finally, we show that $\varphi$ is a surjection. Suppose $Y \in  \wp(U)_{P + Q}$.
Then, $Y = Y_{P + Q}$ and $Y^c = (Y_{P + Q})^c = (Y^c)^{P + Q}$ belongs to $\wp(U)^{P + Q}$ and 
$\varphi(Y^c) = Y^{cc} = Y$, which completes the proof.
\end{proof}

Some lattice-theoretical properties of $\wp(U)^{P + Q}$ and $\wp(U)_{P + Q}$ are considered in \cite{She2012}. 
For example, these lattices are not generally distributive.

\begin{definition}
We say that the equivalences $P$ and $Q$ on $U$ are \emph{coherent} if
\[  (\forall x \in U) \, P(x) \subseteq Q(x) \quad \text{or} \quad Q(x) \subseteq P(x). \]
\end{definition}

The following proposition lists some characteristic properties of coherent equivalences.

\begin{proposition} \label{Prop:CoherenceEquivalent}
Let $P$ and $Q$ be equivalences on $U$. The following are equivalent:
\begin{enumerate}[\rm (a)]
\item $P$ and $Q$ are coherent;
\item $P \cup Q$ is an equivalence;
\item $P \cup Q = P \circ Q$.
\end{enumerate}
\end{proposition}

\begin{proof} (a)$\Rightarrow$(b): Suppose $P$ and $Q$ are coherent and let $(a,b)$ and $(b,c)$
belong to $P \cup Q$. If $(a,b)$ and $(b,c)$ belong to either $P$ or $Q$, there is nothing to prove. Thus,
assume $(a,b) \in P$ and $(b,c) \in Q$. Because $P$ and $Q$ are coherent, $P(b) \subseteq Q(b)$ or 
$Q(b) \subseteq P(b)$. If  $P(b) \subseteq Q(b)$, then $a \in P(b) \subseteq Q(b)$, gives 
$(a,c) \in Q \subseteq P \cup Q$. Analogously, $Q(b) \subseteq P(b)$ implies 
$c \in Q(b) \subseteq P(b)$ and $(a,c) \in P \subseteq P \cup Q$. The case  $(a,b) \in Q$ and $(b,c) \in P$
can be considered similarly.

(b)$\Rightarrow$(c): Suppose $(a,b) \in P \cup Q$. Then  $(a,b) \in P$ or  $(a,b) \in Q$. Because 
$P$ and $Q$ are reflexive, $P,Q \subseteq P \circ Q$ and $(a,b) \in P \circ Q$. Hence, $P \cup Q \subseteq P \circ Q$.
If $(a,b) \in P \circ Q$, then there is $c \in U$ such that $(a,c) \in P$ and $(c,b) \in Q$. This means
$(a,c),(c,b) \in P \cup Q$. Because $P \cup Q$ is transitive, we have $(a,b) \in P \cup Q$.
Thus, also $P \circ Q \subseteq P \cup Q$ holds.

(c)$\Rightarrow$(a): Suppose (c) holds, but $P$ and $Q$ are not coherent. Then there is $x \in U$ such that
$P(x) \nsubseteq Q(x)$ and $Q(x) \nsubseteq P(x)$. This means that there is $a \in P(x)$ such that $a \notin Q(x)$
and $b \in Q(x)$ such that $b \notin P(x)$. We have $(a,x) \in P$ and $(x,b) \in Q$. This yields
$(a,b) \in P \circ Q = P \cup Q$. Therefore, $(a,b) \in P$ or $(a,b) \in Q$. If $(a,b) \in P$, 
then $(a,x) \in P$ gives $b \in P(x)$, a contradiction. If
$(a,b) \in Q$, then $(b,x) \in Q$ gives $a \in Q(x)$, a contradiction again. Thus,  $P$ and $Q$ must be coherent.
\end{proof}

\begin{remark}
The notion of coherence originates in \cite{Schreider75}. There Ju.~A.~Schreider defined that two equivalences $P$
and $Q$ on $U$ are ``coherent'' if there are two disjoint subsets $U_1, U_2 \subseteq U$ (one of which can
be empty), relations $P_1, Q_1$ on $U_1$ and relations $P_2,Q_2$ on $U_2$, which satisfy
\[
U = U_1 \cup U_2, \quad P = P_1 \cup P_2, \quad Q = Q_1 \cup Q_2, \quad P_1 \subseteq Q_1, \quad Q_2 \subseteq P_2.
\]
He proved \cite[Theorem 2.5]{Schreider75} that this condition is equivalent to the fact that the union $P \cup Q$ is an equivalence. 
This means that Schreider's definition coincides with our definition of $P$ and $Q$ being coherent.
\end{remark}

\begin{example} \label{Exa:Coherent}
Let $P$ and $Q$ be coherent. We can divide $U$ into two parts 
\[ U_1 = \{ x \in U \mid Q(x) \subseteq P(x) \} \quad \text{ and } \quad 
U_2 = \{ x \in U \mid P(x) \subset Q(x) \} .\]
So, if $P(x) = Q(x)$, then $x$ belongs $U_1$. Note that $(P \cup Q)(x) = P(x) \cup Q(x)$ for all $x \in U$. 
We have that $P \cup Q$ is an equivalence such that
\[ (P \cup Q)(x) = \left \{ 
\begin{array}{ll}
P(x) & \text{if $x \in U_1$},\\
Q(x) & \text{if $x \in U_2$}. 
\end{array} \right .
\]
Suppose $U = \{a,b,c,d,e\}$. Let $P$ and $Q$ be equivalences such that
\[ U/P = \{ \{a,e\}, \{b\},\{c\}, \{d\} \}  \quad \text{ and } \quad  U/Q = \{ \{a\}, \{b,c\}, \{d\}, \{e\} \}.\] 
Now $U_1 = \{a,d,e\}$, $U_2 = \{b,c\}$, and $U / (P \cup Q) = \{\{a,e\}, \{b,c\},\{d\}\}$.
\end{example}

\begin{proposition} \label{Prop:CohenrentIntersection}
If $P$ and $Q$ are coherent equivalences on $U$, then for all $X \subseteq U$,
\[ X^{P + Q} = X^{P \cap Q} \qquad \text{and} \qquad X_{P + Q} = X_{P \cap Q} . \]
\end{proposition}

\begin{proof}
Because $P \cap Q$ is included in $P$ and $Q$, $X^{P \cap Q}$ is included in $X^P$ and $X^Q$, which 
implies $X^{P \cap Q} \subseteq X^P \cap X^Q = X^{P + Q}$. On the other hand, if $x \in X^{P + Q}$,
then $P(x) \cap X \neq \emptyset$ and $Q(x) \cap X \neq \emptyset$. Because $(P \cap Q)(x) = P(x) \cap Q(x)$,
and by coherency, $P(x) \subseteq Q(x)$ or $Q(x) \subseteq P(x)$, we have that  $(P \cap Q)(x) = P(x)$
or $(P \cap Q)(x) = Q(x)$. This implies $(P \cap Q)(x) \cap X \neq \emptyset$ and $x \in X^{P \cap Q}$.

The other claim follows from the duality of the operator pairs:
\[ X_{P + Q} = \big ( (X^c)^{P + Q} \big )^c =  \big ( (X^c)^{P \cap Q} \big )^c = X_{P \cap Q} . \]
\end{proof}

Because $P \cap Q$ is an equivalence and the properties of approximations determined by equivalences are well known, we have that if 
$P$ and $Q$ are coherent, then 
\[ \mathrm{Def}(P \cap Q) = {\rm Sat}(P \cap Q) = \wp(U)^{P + Q} = \wp(U)_{P + Q} \]
and this family of sets forms a complete atomistic Boolean lattice.

\begin{proposition} \label{Prop:ApproximationsDistributive}
If $\wp(U)^{P + Q}$ is distributive, then $P$ and $Q$ are coherent.
\end{proposition}

\begin{proof}
Assume that  $P$ and $Q$ are not coherent. Then there is $x \in U$ such that $P(x) \nsubseteq Q(x)$ and $Q(x) \nsubseteq P(x)$.
This means that there is $a \in P(x) \setminus Q(x)$ and $b \in Q(x) \setminus P(x)$.

Obviously,
\[
\{x\}^{P + Q} = \{x\}^{P} \cap \{x\}^{Q} = P(x) \cap Q(x) = (P \cap Q)(x). 
\]
In a similar way, we can see $\{a\}^{P + Q} = (P \cap Q)(a)$ and $\{b\}^{P + Q} = (P \cap Q)(b)$. Now in $\wp(U)^{P + Q}$, 
\begin{align*}
& \{x\}^{P + Q} \wedge \big (\{a\}^{P + Q} \vee \{b\}^{P + Q} \big ) = (P \cap Q)(x) \wedge \big ( (P \cap Q)(a) \vee (P \cap Q)(b) \big ) = \\
&  (P \cap Q)(x) \cap \big ( (P \cap Q)(a) \cup (P \cap Q)(b) \big )^{P + Q} = \\
& (P \cap Q)(x) \cap \big ((P \cap Q)(a) \cup (P \cap Q)(b) \big )^{P} \cap \big ((P \cap Q)(a) \cup (P \cap Q)(b) \big)^{Q} = \\
& (P \cap Q)(x) \cap \big ((P \cap Q)(a)^{P} \cup (P \cap Q)(b)^{P} \big) \cap \big ((P \cap Q)(a)^{Q} \cup (P \cap Q)(b)^{Q} \big ).
\end{align*}
Note that $(P \cap Q)(a)^{P} = P(a)$. This is because $(P \cap Q)(a)^{P} \subseteq P(a)^P = P(a)$ and $a \in (P \cap Q)(a)$
implies $P(a) = \{a\}^P \subseteq (P \cap Q)(a)^P$. In a similar manner,
\[
(P \cap Q)(b)^{P} = P(b), \quad (P \cap Q)(a)^{Q} = Q(a), \quad (P \cap Q)(b)^{Q} = Q(b). \] 
Because $P(a) = P(x)$ and $Q(b) = Q(x)$, we have
\begin{align*}
 \{x\}^{P + Q} \wedge (\{a\}^{P + Q} \vee \{b\}^{P + Q}) &= (P \cap Q)(x) \cap ((P(a) \cup P(b)) \cap  (Q(a) \cup Q(b)) \\
 & \supseteq (P \cap Q)(x) \cap (P(a) \cap Q(b)) \\
 &=  (P(x) \cap Q(x)) \cap (P(x) \cap Q(x)) \\
 &=  P(x) \cap Q(x) \neq \emptyset.
\end{align*}
On the other hand, 
\[ (\{x\}^{P + Q} \wedge \{a\}^{P + Q}) \vee (\{x\}^{P + Q} \wedge \{b\}^{P + Q}) = \]
\[ ( (P \cap Q)(x) \cap (P \cap Q)(a) ) \vee ( (P \cap Q)(x) \cap (P \cap Q)(b) ) \]
Because  $a \in P(x) \setminus Q(x)$,  $P(a) \cap Q(x) = \emptyset$. This implies that
\[ (P \cap Q)(x) \cap (P \cap Q)(a) = P(x) \cap Q(x) \cap P(a) \cap Q(a) = \emptyset.\] 
Similarly, $b \in Q(x) \setminus P(x)$ means that $Q(b) \cap P(x)  = \emptyset$, and we have  $(P \cap Q)(x) \cap (P \cap Q)(b) = \emptyset$.
Therefore, 
\[ (\{x\}^{P + Q} \wedge \{a\}^{P + Q}) \vee (\{x\}^{P + Q} \wedge \{b\}^{P + Q}) = \emptyset \vee \emptyset = \emptyset . \]
We have now shown that if $P$ and $Q$ are not coherent, then $\wp(U)^{P + Q}$ cannot be distributive. This means that if
$\wp(U)^{P + Q}$ is distributive, then $P$ and $Q$ are coherent.
\end{proof}

We can now write the following corollary.

\begin{corollary}
Let $P$ and $Q$ be equivalences on $U$. The following are equivalent.
\begin{enumerate}[\rm (a)]
\item $\wp(U)_{P + Q}$ and $\wp(U)^{P + Q}$ are distributive;
\item $P$ and $Q$ are coherent;
\item $\wp(U)_{P + Q}$ and $\wp(U)^{P + Q}$ are completely distributive;
\item $\wp(U)_{P + Q}$ and $\wp(U)^{P + Q}$ are atomistic Boolean lattices.
\end{enumerate}
\end{corollary}

\begin{proof}
If $\wp(U)_{P + Q}$ and $\wp(U)^{P + Q}$ are distributive, then $P$ and $Q$ are coherent
by Proposition~\ref{Prop:ApproximationsDistributive}.
Thus (a) implies (b). On the other hand, if $P$ and $Q$ are coherent, then, by
Proposition~\ref{Prop:CohenrentIntersection}, $X^{P + Q} = X^{P \cap Q}$ and $X_{P + Q} = X_{P \cap Q}$.
Because $P \cap Q$ is an equivalence, $X^{P \cap Q}$ and $X_{P \cap Q}$ are completely distributive complete lattices
and atomistic Boolean lattices. Therefore, (b) implies both (c) and (d). On the other hand,
trivially (c) implies (a), and the same holds for (d). Thus, (a)--(d) are equivalent.
\end{proof}

Let $X \subseteq U$. In \cite{She2012}, the authors introduced the \emph{pessimistic lower approximation} of $X$ by  
\begin{equation} \label{Eq:PessimisticOne}
\{ x \in U \mid P(x) \subseteq X  \text{ \ and \ }  Q(x) \subseteq X \}. 
\end{equation}
It it obvious that \eqref{Eq:PessimisticOne} equals $X_P \cap X_Q$. We can now write the following lemma.
\begin{lemma}
For all $X \subseteq U$,
\[ X_{P \cup Q} = X_P \cap X_Q. \]
\end{lemma}

\begin{proof}
Because $P$ and $Q$ are included in $P \cup Q$, we have that $X_{P \cup Q}$ is included in $X_P$ and $X_Q$, which 
gives that $X_{P \cup Q} \subseteq X_P \cap X_Q$. Conversely, if $x \in X_P \cap X_Q$, then
$P(x)$ and $Q(x)$ are included in $X$. Therefore, $(P \cup Q)(x) = P(x) \cup Q(x) \subseteq X$, and $x \in X_{P \cup Q}$.
\end{proof}

Thus, the ``pessimistic'' lower approximation of $X$ coincides with $X_{P \cup Q}$. In \cite{She2012},
the \emph{pessimistic upper approximation} is defined as the dual of the ``pessimistic'' lower approximation.
We have that the ``pessimistic'' upper approximations of $X$ is $X^{P \cup Q}$. It is also easy to see that
$X^{P \cup Q} = X^{P} \cup X^{Q}$. The different approximations of any $X \subseteq U$ can now be ordered as
\[ X_{P \cup Q} \subseteq X_P, X_Q \subseteq X_{P + Q} \subseteq X \subseteq X^{P + Q} \subseteq X^P,X^Q \subseteq X^{P \cup Q}. \]
The relation $P \cup Q$ is generally a tolerance and it is an equivalence if and only if $P$ and $Q$ are coherent. 
Because $P \cup Q$ is a tolerance, we can write the following proposition.

\begin{proposition} \label{Prop:Ortholattice} Let $P$ and $Q$ be equivalences on $U$.
\begin{enumerate}[\rm (a)]
\item The map $X \mapsto (X^c)^{P \cup Q}$ is an orthocomplementation in $\wp(U)^{P \cup Q}$.
\item The map $X \mapsto (X^c)_{P \cup Q}$ is an orthocomplementation in $\wp(U)_{P \cup Q}$.
\end{enumerate}
\end{proposition}

The lattice  $\wp(U)_{P \cup Q}$  is not necessarily distributive and next our aim is to give a sufficient and necessary condition for 
$\wp(U)_{P \cup Q}$ to be completely distributive lattice.

Next we find out what are the elements $x \in U$ such that $(P \cup Q)(x)$ is a  $P \cup Q$-block.
Let us divide $U$ into three disjoint sets:
\begin{align*}
U_1 &= \{ x \in U \mid Q(x) \subseteq P(x) \}; \\
U_2 &= \{ x \in U \mid P(x) \subset Q(x) \}; \\
U_3 &= \{ x \in U \mid Q(x) \nsubseteq P(x) \ \text{ and } \ P(x) \nsubseteq Q(x) \}.
\end{align*}
Note that, as earlier, if $P(x) = Q(x)$, then $x \in U_1$. We have three different cases:
\begin{enumerate}[(i)]
\item $x \in U_1$: \ $(P \cup Q)(x) = P(x) \cup Q(x) = P(x)$ is a $P \cup Q$-block, because it is a $P \cup Q$-preblock.
\item $x \in U_2$: \ $(P \cup Q)(x) = P(x) \cup Q(x) = Q(x)$ is a $P \cup Q$-block.
\item $x \in U_3$: \ There are elements $a \in P(x) \setminus Q(x)$ and $b \in Q(x) \setminus P(x)$. Assume that 
$(P \cup Q)(x)$ is a $P \cup Q$-block. Then $a,b \in (P \cup Q)(x) = P(x) \cup Q(x)$ means $(a,b) \in P \cup Q$.  
If $(a,b) \in P$, then $(x,a) \in P$ yields $b \in P(x)$, a contraction. Similarly, $(a,b) \in Q$
implies $a \in Q(x)$, a contraction again. Therefore, $(P \cup Q)(x)$ cannot be a $P \cup Q$-block.
\end{enumerate}

By applying (T8) we can state that $P \cup Q$ is a tolerance induced by an irredundant covering if and only if 
\[ \mathcal{H}(P + Q) = \{ P(x) \mid x \in U_1\} \cup  \{ Q(x) \mid x \in U_2\} \]
induces $P \cup Q$. This also means that the irredundant covering inducing $P \cup Q$ is ${\mathcal{H}(P + Q)}$.

Since it is obvious that the tolerance induced by $\mathcal{H}(P + Q)$ is included in $P \cup Q$, we have that
$P \cup Q$ is a tolerance induced by an irredundant covering if and only if $P \cup Q \subseteq T_{\mathcal{H}(P + Q)}$.
Note that if $P$ and $Q$ are coherent, then $U_1 \cup U_2 = U$ and $U_3 = \emptyset$. 
Since in this case $P \cup Q$ is an equivalence according to Proposition~\ref{Prop:CoherenceEquivalent},
$\mathcal{H}(P + Q)$ consists of the equivalence classes of $P \cup Q$,
trivially forming an irredundant covering of $U$.

\begin{example} \label{Exa:Different cases}
(a)  Let $P$ and $Q$ be equivalences on $U = \{a,b,c\}$ such that
\[ U/P = \{\{a,b\},\{c\}\}  \quad \text{ and } \quad U/Q = \{\{a\}, \{b,c\}\}.\]
Then $U_1 = \{a\}$, $U_2 = \{c\}$, $U_3 = \{b\}$. Obviously,
\[ \mathcal{H}(P + Q) = \{ P(a), Q(c) \} = \{\{a,b\},\{b,c\}\} \]
induces $P \cup Q$ and hence $P \cup Q$ is a tolerance induced by an irredundant covering.

(b) If $P$ and $Q$ are equivalences on $U = \{a,b,c,d\}$ such that
\[ U/P = \{\{a,b\},\{c,d\}\}  \quad \text{ and } \quad  U/Q = \{\{a,d\}, \{b,c\}\},\]
then  $U_1 = U_2 = \emptyset$ and $U_3 = U$. Now $\mathcal{H}(P + Q) = \emptyset$
does not induce $P \cup Q$ and therefore $P \cup Q$ is not a tolerance induced by an irredundant covering.

(c) If $P \cup Q$ is a tolerance induced by an irredundant covering, then $\mathcal{H}(P + Q)$ is the covering
which induces the tolerance $P \cup Q$. But even if $\mathcal{H}(P + Q)$ is an irredundant covering, it does not necessarily
induce $P \cup Q$. For instance let $U = \{a,b,c,d\}$, $U / P = \{ \{a\}, \{b,c\}, \{d\}\}$ and $U / Q = \{ \{a,b\}, \{c,d\}\}$.
Then $U_1 = \emptyset$ and $U_2 = \{a,d\}$. Clearly,
\[ \mathcal{H}(P + Q) = \{ Q(a), Q(d) \} = \{\{a,b\},\{c,d\}\} \]
is an irredundant covering, but it does not induce $P \cup Q$. For instance, $(b,c) \in P \cup Q$, but $(b,c)$ does not
belong to the tolerance induced by $\mathcal{H}(P + Q)$.
\end{example}

Based on the results presented in Section~\ref{Sec:IntroPrel}, we can write the following corollary.

\begin{corollary}
If $\mathcal{H}(P + Q)$ induces $P \cup Q$, then $\wp(U)_{P \cup Q}$ and $\wp(U)^{P \cup Q}$ are complete
atomistic Boolean lattices.
\end{corollary}

\begin{example} \label{Exa:ApproximationOrders}
Let $P$ and $Q$ be equivalences on $U = \{a,b,c\}$ of Example~\ref{Exa:Different cases}(a), that is,
\[ U/P = \{\{a,b\}, \{c\}\} \quad \text{ and } \quad  U/Q = \{\{a\}, \{b,c\}\} .\]
In Table~\ref{Table:big} are listed different approximations determined by $P$ and $Q$. 
\begin{table}[ht]
\centering
{\small  
\begin{tabular}{c|cccccccc}
$X$ & $X_P$ & $X_Q$ & $X^P$ & $X^Q$ & $X_{P + Q}$ & $X^{P + Q}$ & $X_{P \cup Q}$  &  $X^{P \cup Q}$ \\ \hline
$\emptyset$ & $\emptyset$ & $\emptyset$ & $\emptyset$ & $\emptyset$ & $\emptyset$ & $\emptyset$ & $\emptyset$ & $\emptyset$ \\
$\{a\}$ & $\emptyset$ & $\{a\}$      & $\{a,b\}$  & $\{a\}$   & $\{a\}$           & $\{a\}$   & $\emptyset$ & $\{a,b\}$\\
$\{b\}$ & $\emptyset$ & $\emptyset$  & $\{a,b\}$  & $\{b,c\}$ & $\emptyset$       & $\{b\}$   & $\emptyset$ & $U$ \\
$\{c\}$ & $\{c\}$  & $\emptyset$     & $\{c\}$    & $\{b,c\}$ & $\{c\}$           & $\{c\}$   & $\emptyset$ & $\{b,c\}$ \\
$\{a,b\}$ & $\{a,b\}$ & $\{a\}$      & $\{a,b\}$  & $U$       & $\{a,b\}$         & $\{a,b\}$ & $\{a\}$  & $U$ \\
$\{a,c\}$ & $\{c\}$ & $\{a\}$        & $U$        & $U$       & $\{a,c\}$         & $U$       & $\emptyset$ & $U$\\
$\{b,c\}$ & $\{c\}$ & $\{b,c\}$      & $U$        & $\{b,c\}$ & $\{b,c\}$         & $\{b,c\}$ & $\{c\}$  & $U$ \\
$U$ & $U$ & $U$ & $U$ & $U$ & $U$ & $U$ & $U$ & $U$ 
\end{tabular}
}
\caption{\label{Table:big} Approximations determined by $P$ and $Q$}
\end{table}

The Boolean lattices
\begin{itemize}
\item $\wp(U)_P = \wp(U)^P = \{ \emptyset, \{a,b\},  \{c\} \}$,
\item $\wp(U)_Q = \wp(U)^Q = \{ \emptyset, \{a\},  \{b,c\} \}$,
\item $\wp(U)_{P \cup Q} = \{ \emptyset, \{a\},  \{c\} \}$, and
\item $\wp(U)^{P \cup Q} = \{ \emptyset, \{a,b\},  \{b,c\} \}$
\end{itemize}
are all isomorphic to the 4-element Boolean lattice $\mathbf{2 \times 2}$. The Hasse diagrams of $\wp(U)_{P + Q}$ and $\wp(U)^{P + Q}$ 
are given in Figure~\ref{Fig:Figure1}.
\begin{figure}[ht]
\centering
\includegraphics[width=90mm]{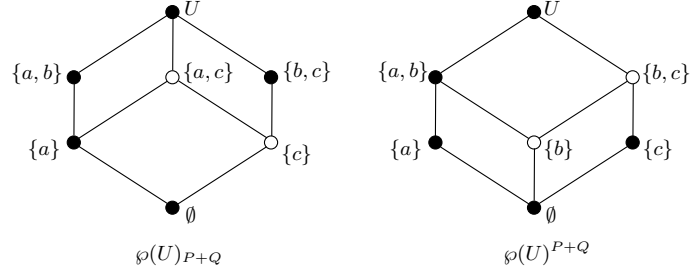}
\caption{The Hasse diagrams of $\wp(U)_{P + Q}$ and $\wp(U)^{P + Q}$. \label{Fig:Figure1}}
\end{figure}

The equivalences $P$ and $Q$ are not coherent, because $P(b) = \{a,b\}$ and $Q(b) = \{b,c\}$ are not $\subseteq$-comparable.
Therefore, the lattices  $\wp(U)_{P + Q}$ and $\wp(U)^{P + Q}$ are not distributive. It is known that
each distributive lattice is modular and that a lattice is modular if and only if it
does not contain the ``pentagon'' $\mathbf{N_5}$ as a sublattice; see e.g. \cite{Grat98} for details. 
The elements of forming  $\mathbf{N_5}$-sublattice are marked with filled circles in Figure~\ref{Fig:Figure1}.
\end{example}

\section{Order structures of multigranular rough sets} \label{Sec:RoughSets}

In this section, we study the lattice-theoretical structure of the set of approximation pairs
\[ \mathit{RS}(P + Q) = \{ (X_{P + Q},X^{P + Q}) \mid X \subseteq U\}, \]
where $P$ and $Q$ are equivalences on $U$.

As noted in \cite{Kong2019}, $\mathit{RS}(P + Q)$ has a polarity defined by
\[   {\sim}(A,B) = ( B^c, A^c)  \]
for $(A,B) \in \mathit{RS}(P + Q)$. Hence, $\mathit{RS}(P + Q)$ is self-dual.

\begin{proposition} \label{Prop:LatticeProperties}
Let $P$ and $Q$ be equivalences on $U$. 
\begin{enumerate}[\rm (i)]
\item If $P$ and $Q$ are coherent, then $\mathit{RS}(P+Q)$ and $\mathit{RS}(P \cup Q)$ are completely
distributive regular double Stone lattices.

\item If the tolerance $P\cup Q$ is induced by the system $\mathcal{H}(P+Q)$ of equivalence classes,
then $RS(P\cup Q)$ is a complete distributive lattice forming a regular pseudocomplemented Kleene algebra.
\end{enumerate}
\end{proposition}

\begin{proof}
(i) Let $P$ and $Q$ be coherent. By Proposition~\ref{Prop:CohenrentIntersection}, $\mathit{RS}(P+Q) = \mathit{RS}(P \cap Q)$.
  By Proposition~\ref{Prop:CoherenceEquivalent}, $P \cup Q$ is an equivalence. The claim now follows from
  the known properties of the rough set lattices determined by an equivalence.

(ii) Since $P\cup Q$ is a tolerance relation, $RS(P\cup Q)$ is a completely distributive lattice if and only if $P \cup Q$ is 
induced by an irredundant covering of $U$. We already noted in the previous section that this holds if and only if
$P\cup Q$ is induced by $\mathcal{H}(P+Q)$. As noted in Section~\ref{Sec:IntroPrel}, in this case $RS(P\cup Q)$
forms a pseudocomplemented Kleene algebra which is regular.
\end{proof}

It is claimed in \cite{Kong2019} that $\mathit{RS}(P + Q)$ is always a distributive and pseudocomplemented lattice (Theorems 7.1 and 7.3).
This is not generally true, as can be seen in our next example.

\begin{example} \label{Exa:RoughSets} Let us continue Example~\ref{Exa:ApproximationOrders}. Now $\mathit{RS}(P + Q)$
consists of the pairs
\[
 \{ (\emptyset,\emptyset), (\{a\},\{a\}), (\emptyset,\{b\}), (\{c\},\{c\}),
 (\{a,b\},\{a,b\}),  (\{a,c\},U),  (\{b,c\},\{b,c\}), (U,U) \}.
\]
The ordered set $\mathit{RS}(P + Q)$  is not a lattice, because,
for instance, $(\{a\},\{a\})$ and $(\emptyset,\{b\})$ have the minimal upper bounds $(\{a,b\},\{a,b\})$ and $(\{a,c\},U)$,
but not a least one. On the other hand,
\[
 \mathit{RS}(P \cup Q) =
 (\emptyset,\emptyset), (\emptyset,\{a,b\}), (\emptyset,\{b,c\}), (\emptyset,U),  (\{a\},U),  (\{c\},U), (U,U) \}
\]
forms a regular double pseudocomplemented lattice and a Kleene algebra. The Hasse diagrams of these ordered sets are
presented in Figure~\ref{Fig:Figure2}. For simplicity, we denote subsets of $U$  by sequences of letters. For instance, $\{a,b\}$ is written as $ab$.
\begin{figure}[h]
\centering
\includegraphics[width=120mm]{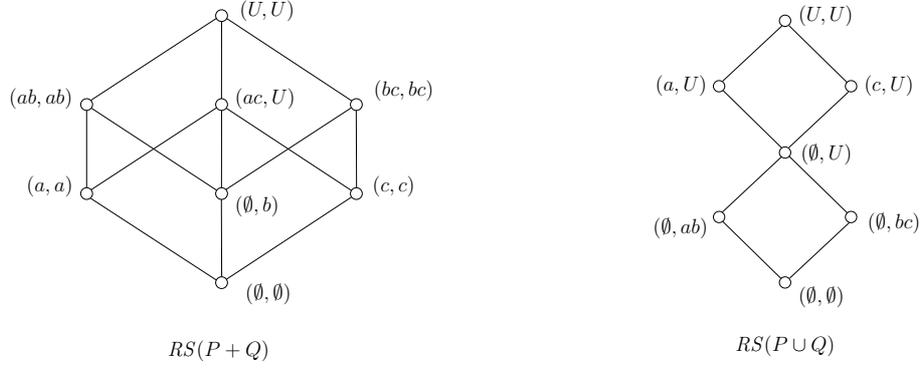}
\caption{The Hasse diagrams of $\mathit{RS}(P + Q)$ and $\mathit{RS}(P \cup Q)$. \label{Fig:Figure2}}
\end{figure}
\end{example}

Let $P,Q\subseteq U\times U$ be equivalence relations. We introduce the condition:
\begin{equation}\label{Eq:Singletons} \tag{C}
 \text{If $|(P \cap Q)  (x)| = 1$, then $|P(x)| = 1$  or $|Q(x)| = 1$.}
\end{equation}

In what follows, we will prove that $RS(P+Q)$ is a complete lattice whenever the equivalences $P$ and $Q$ satisfy condition
(C). For this, we need to consider first some elementary facts. The following lemma is well known; see \cite[Lemma 2.6]{JaKoRa2019}, for instance.

\begin{lemma} \label{Lem:Subsets}
If $R$ and $S$ are equivalences on $U$ such that $R \subseteq S$, then for any
$X\subseteq U$,
\[ \left(  X^{R} \right)^{S} = \left(  X^{S}\right)^{R}=X^{S} . \]
\end{lemma}

\begin{corollary} \label{Cor:Intersections}
Let $P$ and $Q$ be equivalences on $U$. The following facts hold for any $X,Y \subseteq U$:
\begin{enumerate}[\rm (i)]
\item $\left ( X^{P\cap Q} \right)^{P+Q} = X^{P+Q}$;

\item $X^{P\cap Q} = Y^{P\cap Q}$  implies $X^{P+Q} = Y^{P+Q}$.
\end{enumerate}
\end{corollary}

\begin{proof} 
(i) We have $\left(X^{P\cap Q}\right)^{P+Q} =\left (X^{P\cap Q}\right)^{P} \cap \left( X^{P\cap Q}\right)^{Q}$.
Since $P\cap Q\subseteq P,Q$, we get 
$\left(  X^{P \cap Q}\right)^{P} = X^{P}$ and $\left( X^{P\cap Q}\right)^{Q}=X^{Q}$ by Lemma~\ref{Lem:Subsets}.  
Thus, we obtain $\left( X^{P\cap Q}\right)^{P+Q} = X^{P}\cap X^{Q} = X^{P+Q}$.

(ii) If $X^{P\cap Q} = Y^{P\cap Q}$, then 
$X^{P+Q} = \left( X^{P\cap Q} \right)^{P+Q} = \left( Y^{P\cap Q}\right)^{P+Q} = Y^{P+Q}$ by (i).
\end{proof}

\begin{lemma} \label{Lem:OptDefinable}
The elements of $\wp(U)^{P+Q}$ and $\wp(U)_{P+Q}$ are $P\cap Q$-definable.
\end{lemma}

\begin{proof}
Take any $X \in \wp(U)^{P+Q}$. Then $X \subseteq X^{P\cap Q}\subseteq \left ( X^{P\cap Q} \right )^{P+Q}=X^{P+Q} = X$ 
yields $X^{P\cap Q}=X$, that is, $X$ is $P\cap Q$-definable. On the other hand, if $X\in\mathcal{P}(U)_{P+Q}$, 
then $X = X_{P+Q} = X_{P} \cup X_{Q}$ and we have
\begin{align*}
X \subseteq X^{P \cap Q} &=  \left( X_{P} \cup X_{Q} \right)^{P \cap Q} =
\left( X_{P} \right )^{P\cap Q} \cup \left( X_{Q} \right )^{P\cap Q} 
\subseteq \left( X_{P} \right)^{P} \cup \left( X_{Q} \right)^{Q} \\ 
& = X_P \cup X_Q \subseteq X.
\end{align*}
This yields $X^{P\cap Q}=X$,  finishing our proof. 
\end{proof}

\begin{theorem} \label{Thm:CompleteLattice}
Let $P$ and $Q$ be two equivalences on $U$ satisfying \eqref{Eq:Singletons}. Then
$\mathit{RS}(P+Q)$ is a complete lattice such that 
\begin{align}
\bigwedge_{X\in\mathcal{H}} (X_{P+Q},X^{P+Q}) &= \label{Eq:LatticeMeet}
\Big(  \big( \bigcap_{X\in\mathcal{H}} X_{P+Q} \big)_{P+Q}, \bigcap_{X\in\mathcal{H}} X^{P+Q} \Big) \\
\intertext{and}
\bigvee_{X \in\mathcal{H}} (X_{P+Q},X^{P+Q}) &= \label{Eq:LatticeJoin}
\Big(  \bigcup_{X\in\mathcal{H}} X_{P+Q} , \big(  \bigcup_{X\in\mathcal{H}} X^{P+Q} \big)^{P+Q} \Big) 
\end{align}
for all $\mathcal{H} \subseteq \wp(U)$.
\end{theorem}

\begin{proof}
First, we prove equation~\eqref{Eq:LatticeMeet}. 
Clearly, $\left( \bigcap \{ X_{P+Q} \mid X \in \mathcal{H}\} \right)_{P+Q}$ belongs to $\wp(U)_{P+Q}$ 
and $\bigcap \{ X^{P+Q} \mid X \in \mathcal{H}\}$ is in $\wp(U)^{P+Q}$. 
We prove that their pair forms a rough set, that is, we are going to construct a set $Z\subseteq U$ with 
\[ Z_{P+Q} =  \big( \bigcap_{X\in\mathcal{H}} X_{P+Q} \big)_{P+Q} \text{ \ and \ } \
Z^{P+Q} = \bigcap_{X\in\mathcal{H}} X^{P+Q}.\]
By Lemma~\ref{Lem:OptDefinable}, $X_{P+Q}$ and $X^{P+Q}$ are $P\cap Q$-definable, 
whence $\bigcap \{ X_{P+Q} \mid X \in \mathcal{H}\}$ and 
$\bigcap \{ X^{P+Q} \mid X \in \mathcal{H}\}$ are also $P\cap Q$-definable sets. 
Hence, their difference 
\[
D := \Big (  \bigcap_{X\in\mathcal{H}} X^{P+Q} \Big) \setminus 
\left( \bigcap_{X\in\mathcal{H}} X_{P+Q} \right) \]
is also  $P\cap Q$-definable.
We claim that for any element $x\in D$, $|( P\cap Q)(x)| \geq 2$. 
Indeed, suppose $|(P \cap Q)(x)| = 1$. Then condition (C) gives 
$|P(x)| = 1$ or $|Q(x)|=1$. Let us consider the case $|P(x)|=1$. Now $x\in X^{P+Q} = X^{P}\cap X^{Q} $ for
all $X\in \mathcal{H}$. We get $P(x)=\{x\}\subseteq X$ for any $X\in\mathcal{H}$, and this
implies $x\in X_{P}\subseteq X_{P}\cup X_{Q}=X_{P+Q}$, for all $X\in
\mathcal{H}$. Thus we obtain $x \in \bigcap \{ X_{P+Q} \mid X\in\mathcal{H}\}$,
a contradiction. Similarly, we may show that $|Q(x)|=1$ is not possible.
Hence,  $|( P\cap Q)(x)| \geq 2$ for each $x\in D$.

Let $e$ be the restriction of $P \cap Q$ into $D$. Then, as we just noted, each $e$-class has at least two
elements. By the Axiom of Choice, there is a function $f \colon D/e \to D$ which picks
from each $D/e$-class $\beta$ one element $f(\beta)$ of $D$. Let $\Gamma$ be the image set of $f$, that is,
\[ \Gamma := \{ f(\beta) \mid \beta \in D/e \}.\]
It is now clear that $\Gamma^{P \cap Q} = D$. Let us define 
\[
Z := \Big ( \bigcap_{X \in \mathcal{H}} X_{P+Q} \Big ) \cup \Gamma.
\]
By the definition of $D$, we have $\bigcap \{ X_{P+Q} \mid X\in\mathcal{H}\} \cup D = \bigcap \{ X^{P+Q} \mid X\in\mathcal{H}\}$.
Therefore, by Corollary~\ref{Cor:Intersections},
\begin{align*}
Z^{P+Q} &= \big ( Z^{P\cap Q}\big)^{P+Q} = \Big ( \Big ( \bigcap_{X \in \mathcal{H}} X_{P+Q} \Big ) \cup \Gamma\Big)^{P\cap Q}\Big)^{P+Q} \\
& = \Big (  \big(  \bigcap_{X\in\mathcal{H}} X_{P+Q} \big)^{P\cap Q} \cup \Gamma^{P\cap Q} \Big)^{P+Q} 
= \Big (  \big ( \bigcap_{X\in\mathcal{H}} X_{P+Q} \big )  \cup D \Big )^{P+Q} \\
& = \big ( \bigcap_{X\in\mathcal{H}} X^{P+Q} \big)^{P+Q} = \bigcap_{X\in\mathcal{H}} X^{P+Q}.
\end{align*}

We now prove the equality for the other component. Because $\bigcap_{X\in\mathcal{H}} X_{P+Q} \subseteq Z$, 
we have $\big( \bigcap_{X\in\mathcal{H}} X_{P+Q} \big)_{P+Q} \subseteq Z_{P+Q}$.
In order to prove the converse inclusion, let $x \in Z_{P + Q}$. This means that $P(x) \subseteq Z$ or $Q(x) \subseteq Z$.
We show that $P(x) \subseteq Z = \bigcap \{ X_{P + Q} \mid X \in \mathcal{H}\} \cup \Gamma$ implies
$P(x) \subseteq \bigcap \{ X_{P + Q} \mid X \in \mathcal{H}\}$, which is equivalent to $P(x) \cap \Gamma = \emptyset$.

Suppose that $P(x) \subseteq Z$ and $P(x) \cap \Gamma \neq \emptyset$. There is $y \in \Gamma$ such that $(x,y) \in P$.
Because 
\[ y \in \Gamma \subseteq \Gamma^{P \cap Q} = (D \setminus \Gamma)^{P \cap Q},\]
there exists an element $z \in D \setminus \Gamma$ such that $(y,z) \in P \cap Q \subseteq P$. We have that
$(x,z) \in P$ and $z \in P(x) \subseteq Z$. This is impossible since
\begin{align*}
Z \cap (D \setminus \Gamma) & = \Big ( \big (\bigcap_{X \in \mathcal{H}} X_{P + Q} \big ) \cup \Gamma \Big ) \cap (D \setminus \Gamma)  \\
& =\Big ( \big (\bigcap_{X \in \mathcal{H}} X_{P + Q} \big ) \cap (D \setminus \Gamma) \Big )
        \cup \underbrace{\big (\Gamma \cap (D \setminus \Gamma)\big)}_{\emptyset} \\
& \subseteq  \Big (\bigcap_{X \in \mathcal{H}} X_{P + Q} \Big ) \cap D = \emptyset .
\end{align*}
Thus, $P(x) \subseteq \bigcap \{ X_{P + Q} \mid X \in \mathcal{H}\}$.
Similarly, we can show that  $Q(x) \subseteq Z$ implies $Q(x) \subseteq \bigcap \{ X_{P + Q} \mid X \in \mathcal{H}\}$. 
Therefore, $x \in \big (\bigcap \{ X_{P + Q} \mid X \in \mathcal{H}\}\big )_{P + Q}$ and we have proved 
that
\[ Z_{P+Q} =  \big( \bigcap_{X\in\mathcal{H}} X_{P+Q} \big)_{P+Q}.\] 
This means that \eqref{Eq:LatticeMeet} holds.

Equation \eqref{Eq:LatticeMeet} says that $\textit{RS}(P+Q)$ is a complete meet-semilattice. 
Since $\textit{RS}(P+Q)$ is self-dual by the map $\sim$, it is also a complete join-semilattice, and hence it is a complete lattice. 
Then for any $\mathcal{H} \subseteq \wp(U)$, the join 
\[ 
\bigvee_{X \in\mathcal{H}} (X_{P+Q},X^{P+Q})
\]
exists in $(\mathit{RS}(P+Q)$ and 
\begin{align*}
\bigvee_{X\in\mathcal{H}} (X_{P+Q},X^{P+Q}) &=  {\sim} \Big( \bigwedge_{X\in\mathcal{H}} {\sim} (X_{P+Q},X^{P+Q} ) \Big)  \\
&= {\sim} \Big( \bigwedge_{X\in\mathcal{H}} \big ( ( X^{P+Q})^c, (X_{P+Q})^c \big)\Big) \\
&= {\sim} \Big (  \big(  \bigcap_{X\in\mathcal{H}} (  X^{P+Q})^{c} \big )_{P+Q}, \bigcap_{X\in\mathcal{H}} \big ( X_{P+Q}\big)^{c} \Big) \\
&= \Big(  \big(  \bigcap_{X\in\mathcal{H}} \left(  X_{P+Q} \right)^{c} \big)^c, 
\big(  \big ( \bigcap_{X\in\mathcal{H}} ( X^{P+Q} )^c \big)_{P+Q}\big)^{c}\Big)  \\
&= \Big(   \bigcup_{X\in\mathcal{H}}  X_{P+Q} , 
\big(  \bigcup_{X\in\mathcal{H}}  X^{P+Q} \big)^{P+Q} \Big),  \\
\end{align*}
proving \eqref{Eq:LatticeJoin}.
\end{proof}

\begin{example}
There exist equivalences $P$ and $Q$ such that  $\mathit{RS}(P + Q)$ is a nondistributive lattice.
As we shall see in Remark~\ref{Rem:Nondistributive}, for a finite $U$,  $\mathit{RS}(P + Q)$ is not
distributive if $P$ and $Q$ are not coherent. Therefore, it is enough to find noncoherent equivalences $P$ and $Q$
on a finite set satisfying (C) by Theorem~\ref{Thm:CompleteLattice}.

Let $P$ and $Q$ be equivalences on $U = \{1,2,3,4\}$ such that 
\[
U/P = \{ \{1,2,3\}, \{4\}\} \quad \text{ and } \quad U/Q = \{ \{1\}, \{2, 3, 4\}\}.
\]
Now $P$ and $Q$ are not coherent, but (C) is satisfied.

The lattice  $\mathit{RS}(P + Q)$ is depicted in Figure~\ref{Fig:nondistributive}.
A lattice is said to be \emph{semimodular} (or \emph{upper semimodular}) if
$a \wedge b \prec a$ implies $b \prec a \vee b$. The dual property is called \emph{lower semimodular}.
Modular lattices are both upper and lower semimodular \cite{Birk95}. As we already noted, distributive lattices are modular.
If we select $a = (\{1\},\{1\})$ and $b = (\{4\},\{4\})$, we see that $\mathit{RS}(P + Q)$ is not upper semimodular.
Because $\mathit{RS}(P + Q)$ is self-dual, it is not lower semimodular either.

Additionally, $\mathit{RS}(P + Q)$  is not pseudocomplemented, because $(\emptyset,\{2, 3\})$ does not have a pseudocomplement.
\begin{figure}
\centering
\includegraphics[width=60mm]{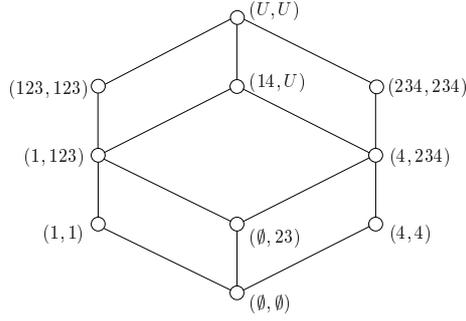}
\caption{The Hasse diagram of nondistributive $\mathit{RS}(P + Q)$ \label{Fig:nondistributive}}
\end{figure}
\end{example}

\section{The smallest completion of $\mathit{RS}(P + Q)$} \label{Sec:Completion}

Let $P$ and $Q$ be equivalences on $U$. We denote by
\[ \Sigma_P = \{ P(x) \mid P(x) = \{x\} \} \qquad \text{and} \qquad \Sigma_Q = \{ Q(x) \mid Q(x) = \{x\} \} \]
the collections of singleton equivalence classes of $P$ and $Q$, respectively. In addition, we define
\[ \mathit{IRS}(P + Q) = \{ (A,B) \in \wp(U)_{P + Q} \times \wp(U)^{P + Q} \mid A \subseteq B \ \text{and} \
   (B \setminus A) \cap (\Sigma_P \cup \Sigma_Q) = \emptyset \}. \]
The set $\mathit{IRS}(P + Q)$ is called the \emph{increasing representation of $P + Q$-rough sets}.

\begin{lemma} \label{Lem:Inclusion}  $\mathit{RS}(P + Q) \subseteq \mathit{IRS}(P + Q)$.
\end{lemma}

\begin{proof} Any element of $\mathit{RS}(P + Q)$ has the form $(X_{P + Q},X^{P + Q})$ for some $X \subseteq U$.
Therefore, $(X_{P + Q},X^{P + Q})$ belongs to $\wp(U)_{P + Q} \times \wp(U)^{P + Q}$ and $X_{P + Q} \subseteq X^{P + Q}$.

We show that $(X^{P + Q} \setminus X_{P + Q}) \cap (\Sigma_P \cup \Sigma_Q) = \emptyset$. Assume, by contradiction, that there
exists an element $x \in (X^{P + Q} \setminus X_{P + Q}) \cap \Sigma_P$. This means that $P(x) = \{x\}$. Therefore,
$x \in X^{P + Q}$ implies $x \in X_{P + Q}$, a contradiction. Similarly, we can show that existence of an element
in $(X^{P + Q} \setminus X_{P + Q}) \cap \Sigma_Q$ is impossible. Hence,
$(X^{P + Q} \setminus X_{P + Q}) \cap (\Sigma_P \cup \Sigma_Q) = \emptyset$. Now we have shown that
$(X_{P + Q},X^{P + Q})$ belongs to $\mathit{IRS}(P + Q)$.
\end{proof}

Let $I$ be an arbitrary index set. For a family $\{L_i\}_{i \in I}$ of complete lattices, the
\emph{direct product} $\prod_{i \in I} L_i$ is a complete lattice with respect to the componentwise order
\[ (x_i)_{i \in I} \leq (y_i)_{i \in I} \iff x_i \leq y_i \quad \text{for all $i \in I$.} \]
The joins and meets are formed componentwise. The lattices $L_i$ are the \emph{factors} of the product.
The maps
\[ \pi_k \big ((x_i)_{i \in I} \big ) = x_k, \quad \text{for $k \in I$,}  \]
are the \emph{canonical projections}. The canonical projections are surjective complete
homomorphisms, that is, they preserve all joins and meets.

Hence, the direct product $\wp(U)_{P + Q} \times \wp(U)^{P + Q}$ is ordered by
\[ (A,B) \leq (C,D) \iff A \subseteq C \quad \text{and} \quad B \subseteq D\]
and it forms a complete lattice such that 
\[
\bigvee_{i \in I} (A_i,B_i) = \Big ( \bigcup_{i \in I} A_i,  \big ( \bigcup_{i \in I} B_i \big)^{P + Q} \Big)
\qquad  \text{and} \qquad
\bigwedge_{i \in I} (A_i,B_i) = \Big ( \big (\bigcap_{i \in I} A_i \big)_{P + Q},  \bigcap_{i \in I} B_i \Big)
\]
for all $\{(A_i,B_i) \mid i \in I\} \subseteq \wp(U)_{P + Q} \times \wp(U)^{P + Q}$. 
The maps
\[
\pi_1 \colon (A,B) \mapsto A \quad \text{and} \quad \pi_2 \colon (A,B) \mapsto B 
\]
are the corresponding canonical projections. The product  $\wp(U)_{P + Q} \times \wp(U)^{P + Q}$
has a polarity defined by
\[ {\sim} \colon (A,B) \mapsto (B^c,A^c). \]
Indeed, if $(A,B)$ and $(C,D)$ are elements of  $\wp(U)_{P + Q} \times \wp(U)^{P + Q}$, then
\[ {\sim}{\sim}(A,B) = {\sim}(B^c,A^c) = (A^{cc},B^{cc}) = (A,B),\]
and
\begin{align*}
(A,B) \leq (C,D) &\iff A \subseteq C \quad \text{and} \quad B \subseteq D
\iff C^c \subseteq A^c \quad \text{and} \quad D^c \subseteq B^c \\
     & \iff (D^c,C^c) \leq (B^c,A^c) \iff  {\sim}(C,D) \leq {\sim}(A,B). 
\end{align*} 

A \emph{complete subdirect product} of complete lattices is a complete sublattice of the direct product
for which the canonical projections onto the factors are all surjective; see \cite{ganter1999formal}, for example.

\begin{proposition} \label{Prop:IRSC_Lattice}
$\mathit{IRS}(P + Q)$ is a complete subdirect product of $\wp(U)_{P+Q}$ and $\wp(U)^{P+Q}$.
\end{proposition}

\begin{proof}
Let $\{(A_i,B_i) \mid i \in I\} \subseteq \mathit{IRS}(P + Q)$.
First we show that $( \bigcup_{i \in I} A_i, ( \bigcup_{i \in I} B_i )^{P + Q} )$ belongs to $\mathit{IRS}(P + Q)$.
As $A_i$ belongs to $\wp(U)_{P + Q}$ for all $i \in I$, we have $\bigcup_{i \in I} A_i \in \wp(U)_{P + Q}$.
Obviously, $(\bigcup_{i \in I} B_i )^{P + Q}  \in \wp(U)^{P + Q}$. It is also clear that
$\bigcup_{i \in I} A_i \subseteq \bigcup_{i \in I} B_i \subseteq (\bigcup_{i \in I} B_i )^{P + Q}$. Let us
assume by contradiction that there exists an element $x \in \Sigma_P$ such that
\[ x \in  \big ( \bigcup_{i \in I} B_i \big )^{P + Q} \mathbin{\big\backslash} \big (\bigcup_{i \in I} A_i \big ).\]
Then $P(x) = \{x\}$ and $x \in ( \bigcup_{i \in I} B_i )^{P} \cap  ( \bigcup_{i \in I} B_i )^{Q}$ gives
$x \in  ( \bigcup_{i \in I} B_i )^{P} =  \bigcup_{i \in I} {B_i}^{P}$. Hence, $P(x) \cap B_k$ for some $k \in I$, which
yields $x \in B_k$. Because $B_k = (B_k \setminus A_k) \cup A_k$ and $x \notin B_k \setminus A_k$ by the definition of $\mathit{IRS}(P + Q)$,
we obtain $x \in A_k \subseteq \bigcup_{i \in I} A_i$, a contradiction. The case $x \in \Sigma_Q$ is analogous. We may deduce
\[ \Big ( \big ( \bigcup_{i \in I} B_i \big )^{P + Q} \mathbin{\big\backslash} \big (\bigcup_{i \in I} A_i \big ) \Big )
\cap (\Sigma_P \cup \Sigma_Q) = \emptyset.\]
We have now proved that $( \bigcup_{i \in I} A_i, ( \bigcup_{i \in I} B_i )^{P + Q} ) \in \mathit{IRS}(P + Q)$.

It is obvious that $( \bigcup_{i \in I} A_i, ( \bigcup_{i \in I} B_i )^{P + Q} )$ is an upper bound for
$\{(A_i,B_i) \mid i \in I\}$. On the other hand, assume that $(X,Y) \in  \mathit{IRS}(P + Q)$ is an upper bound for
$\{(A_i,B_i) \mid i \in I\}$. Then $\bigcup_{i \in I}  A_i \subseteq X$ and $\bigcup_{i \in I}  B_i \subseteq Y$. The latter implies
that $(\bigcup_{i \in I}  B_i)^{P + Q} \subseteq Y^{P + Q} = Y$, because $Y \in \wp(U)^{P + Q}$. Thus,
$( \bigcup_{i \in I} A_i, ( \bigcup_{i \in I} B_i )^{P + Q} )$ is the least  upper bound of $\{(A_i,B_i) \mid i \in I\}$.
The equality for meet can be proved analogously.

Finally, the canonical projections $\pi_1$ and $\pi_2$ restricted to $\mathit{IRS}(P + Q)$ are surjective.
For instance, if $X_{P+Q} \in \wp(U)_{P+Q}$, the pair $(X_{P+Q},X^{P+Q})$ belongs to $\mathit{IRS}(P + Q)$ by
Lemma~\ref{Lem:Inclusion}, and $X_{P+Q}$ is its $\pi_1$-image.
\end{proof}

\begin{proposition} \label{Prop:CompletionPolar}
$\mathit{IRS}(P + Q)$ is a complete polarity sublattice of $\wp(U)_{P + Q} \times \wp(U)^{P + Q}$.
\end{proposition}

\begin{proof}
We have already shown that $\mathit{IRS}(P + Q)$ is a complete sublattice of $\wp(U)_{P + Q} \times \wp(U)^{P + Q}$.
We need to show that the polarity $\sim$ of $\wp(U)_{P + Q} \times \wp(U)^{P + Q}$ is also a polarity of $\mathit{IRS}(P + Q)$.

Let $(A,B) \in \mathit{IRS}(P + Q)$. Then $A \subseteq B$, $A = A_{P + Q}$, $B = B^{P+Q}$,
and $(B \setminus A) \cap (\Sigma_P \cup \Sigma_Q) = \emptyset$.
Now $B^c = (B^{P+Q})^c = (B^c)_{P + Q}$, that is, $B^c \in \wp(U)_{P + Q}$. Similarly, $A^c = (A_{P + Q})^c = (A^c)^{P + Q}$, meaning that
$A^c \in \wp(U)^{P + Q}$. It is obvious that $B^c \subseteq A^c$ since $A \subseteq B$. Finally,
\[ (A^c \setminus B^c) = A^c \cap B^{cc} = B \cap A^c = B \setminus A.\]
Thus,
\[  (A^c \setminus B^c) \cap  (\Sigma_P \cup \Sigma_Q) =  (B \setminus A) \cap (\Sigma_P \cup \Sigma_Q) = \emptyset. \]
We have now shown that $(B^c,A^c)$ belongs to  $\mathit{IRS}(P + Q)$.
\end{proof}

Next we recall some definitions from \cite{DaPr02}; see also \cite{schroder2003}. Let $P$
be an ordered set. If $L$ is a complete lattice such that there is an order-embedding
$\varphi \colon P \to L$, then $L$ is called a \emph{completion} of $P$ (via $\varphi$).
For any $A \subseteq P$, we define
\[ A^u = \{x \in P \mid (\forall a \in A) \, a \leq x\} \quad \text{and} \quad
A^l = \{x \in P \mid (\forall a \in A) \, a \geq x\}. \]
In addition,
\[ \mathrm{DM}(P) = \{ A \subseteq P \mid A^{ul} = A\} .\]
The ordered set $(\mathrm{DM}(P),\subseteq)$ is a complete lattice, known as the
\emph{Dedekind--MacNeille completion} of $P$.

Let $x \in P$. We denote ${\downarrow}x = \{y \in P \mid y \leq x\}$ and define an order-embedding
$\Phi \colon P \to  \mathrm{DM}(P)$ by $\Phi(x) = {\downarrow}x$. Then 
$\mathrm{DM}(P)$ is a completion of $P$ via the map $\Phi$ preserving all joins and
meets which exist in $P$.

A set $Q \subseteq P$ is called \emph{join-dense} in $P$ if for every element $a \in P$, there is a subset $A$ of $Q$
such that $a = \bigvee_P A$. The dual of join-dense is \emph{meet-dense}. Now $\Phi(P)$ is both join-dense
and meet-dense in  $\mathrm{DM}(P)$. In addition, if $L$ is a complete lattice and $P$ is a subset of $L$ which
is both join-dense and meet-dense in $L$, then $L$ is isomorphic to $\mathrm{DM}(P)$.

The following theorem from \cite{schroder2003} verifies that the Dedekind--MacNeille completion is the smallest
completion by showing that every completion of an ordered set contains a copy of the Dedekind--MacNeille completion.

\begin{theorem}
Let $P$ be an ordered set and let $L$ be a completion of $P$ via order-embedding $\varphi \colon P \to L$.
Then there is an order-embedding $\Psi \colon \mathrm{DM}(P) \to L$ such that $\varphi = \Psi \circ \Phi$.
\end{theorem}

As we already pointed out, $\mathit{IRS}(P + Q)$ is a complete lattice containing $\mathit{RS}(P + Q)$, that is,
\mbox{$\mathit{IRS}(P + Q)$} is a completion of $\mathit{RS}(P + Q)$ via the identity mapping. To prove that
$\mathit{IRS}(P + Q)$ is the smallest completion of  $\mathit{RS}(P + Q)$, we need to show that
$\mathit{RS}(P + Q)$ is join-dense and meet-dense in $\mathit{IRS}(P + Q)$.

\begin{theorem} \label{Thm:smallest completion}
$\mathit{IRS}(P + Q)$ is the smallest completion of $\mathit{RS}(P + Q)$.
\end{theorem}

\begin{proof} It is enough to prove that $\mathit{RS}(P + Q)$ is join-dense in $\mathit{IRS}(P + Q)$,
because these two ordered sets are self-dual by the map $\sim$. Indeed, if  $\mathit{RS}(P + Q)$ is join-dense in
$\mathit{IRS}(P + Q)$ and $(A,B) \in \mathit{IRS}(P + Q)$, then ${\sim}(A,B) = \bigvee_{i \in I} (X_i,Y_i)$ for some
$\{ (X_i,Y_i) \mid i \in I \} \subseteq \mathit{RS}(P + Q)$. Now
\[ (A,B) = {\sim} \bigvee_{i \in I} (X_i,Y_i)  = \bigwedge_{i \in I} {\sim} (X_i,Y_i).\]
Because ${\sim} (X_i,Y_i) = ({Y_i}^c,{X_i}^c)$ belongs to $\mathit{RS}(P + Q)$ for each $i \in I$,
we conclude that $\mathit{RS}(P + Q)$ is meet-dense in $\mathit{IRS}(P + Q)$.

In order to prove that $\mathit{RS}(P + Q)$ is join-dense in $\mathit{IRS}(P + Q)$, let $(A,B) \in \mathit{IRS}(P + Q)$.
We consider the set of pairs
\[ \mathcal{H} = \{ (X_{P + Q}, X^{P + Q}) \in \mathit{RS}(P + Q) \mid (X_{P + Q}, X^{P + Q}) \leq (A,B) \}. \]
Clearly $(A,B)$ is an upper bound for $\mathcal{H}$, so $\bigvee \mathcal{H} \leq (A,B)$. We need to show that also
$(A,B) \leq \bigvee \mathcal{H}$. 

Let $x \in A = A_{P + Q} = A_P \cup A_Q$. If $x \in A_P$, then $P(x) \subseteq A \subseteq B$ and 
$P(x)_{P + Q} \subseteq A_{P + Q} = A$ and $P(x)^{P + Q} \subseteq B^{P + Q} = B$. This means that
the rough set $(P(x)_{P + Q}, P(x)^{P + Q})$ belongs to $\mathcal{H}$. In addition, $x \in P(x)_P \subseteq P(x)_{P + Q}$
implies that
\[ x \in P(x)_{P + Q} \subseteq \bigcup \{ X_{P+Q} \mid ( X_{P+Q},  X^{P+Q}) \in \mathcal{H} \}.\]
The case $x \in A_Q$ is analogous. Hence,  $A = \bigcup \{ X_{P+Q} \mid ( X_{P+Q},  X^{P+Q}) \in \mathcal{H} \}$.

Let $y \in B = B^{P+Q}$. Then $\{y\}^{P + Q} \subseteq B^{P+Q} = B$. If $\{y\}_{P+Q} = \emptyset$, then
$(\{y\}_{P+Q},\{y\}^{P+Q}) \leq (A,B)$ and the rough set $(\{y\}_{P+Q},\{y\}^{P+Q})$ belongs to $\mathcal{H}$. Thus,
\begin{align*}
y \in \{y\}^{P+Q} 
& \subseteq \bigcup \{ X^{P + Q} \mid (X_{P + Q}, X^{P + Q}) \in \mathcal{H} \} \\
& \subseteq \Big ( \bigcup \{ X^{P + Q} \mid (X_{P + Q}, X^{P + Q}) \in \mathcal{H} \} \Big )^{P + Q}. 
\end{align*}

If $\{y\}_{P + Q} = \{y\}_P \cup \{y\}_Q \neq \emptyset$, then $P(y) = \{y\}$ or $Q(y) = \{y\}$, meaning that
$y \in \Sigma_P \cup \Sigma_Q$. Because $(A,B) \in \mathit{IRS}(P + Q)$, $A \subseteq B$ and
$(B \setminus A) \cap (\Sigma_P \cup \Sigma_Q) = \emptyset$. This gives that $y \notin B \setminus A$. Because
$B = A \cup (B \setminus A)$, we have that
\[ y \in A = \bigcup \{ X_{P+Q} \mid ( X_{P+Q},  X^{P+Q}) \in \mathcal{H} \} \subseteq
   \Big ( \bigcup \{ X^{P + Q} \mid (X_{P + Q}, X^{P + Q}) \in \mathcal{H} \} \Big )^{P + Q}. \]
We have now proved that also
\[ B = \Big ( \bigcup \{ X^{P + Q} \mid (X_{P + Q}, X^{P + Q}) \in \mathcal{H} \} \Big )^{P + Q} \]
holds. Therefore, $(A,B) = \bigvee \mathcal{H}$ and the proof is completed.
\end{proof}

We end this work with some consequences of Theorem~\ref{Thm:smallest completion}.

\begin{corollary} The following are equivalent:
\begin{enumerate}[\rm (a)]
\item $\mathit{RS}(P + Q)$ is a complete lattice;
\item $\mathit{RS}(P + Q)$ is a complete polarity sublattice of $\wp(U)_{P + Q} \times \wp(U)^{P + Q}$;
\item $\mathit{RS}(P + Q)$ is a complete subdirect product of $\wp(U)_{P + Q}$ and $\wp(U)^{P + Q}$.
\end{enumerate}
\end{corollary}

\begin{proof}
If (a) holds, that is, $\mathit{RS}(P + Q)$ is a complete lattice, then  $\mathit{RS}(P + Q)$ equals $\mathit{IRS}(P + Q)$.
From Propositions \ref{Prop:IRSC_Lattice} and \ref{Prop:CompletionPolar} we obtain (b) and (c).
Trivially, (b) and (c) both separately imply (a).
\end{proof}

Related to coherence and distributivity we can present the following equivalent conditions.

\begin{proposition} \label{Prop:Distributivity}
The following are equivalent.
\begin{enumerate}[\rm (a)]
\item $\mathit{IRS}(P + Q)$ is distributive;
\item $\wp(U)_{P + Q}$ and $\wp(U)^{P + Q}$ are distributive;
\item $P$ and $Q$ are coherent;
\item $\mathit{RS}(P + Q)$ is completely distributive regular double Stone lattice.
\end{enumerate}   
\end{proposition}

\begin{proof}
Because $\wp(U)_{P + Q}$ and $\wp(U)^{P + Q}$ are homomorphic images of  $\mathit{IRS}(P + Q)$,
they are distributive whenever  $\mathit{IRS}(P + Q)$ is, so (a) implies (b). That (b) implies (c)
is clear by Proposition~\ref{Prop:ApproximationsDistributive}.
If $P$ and $Q$ are coherent, then $RS(P + Q) = RS(P \cap Q)$ and therefore (c) implies (d)
by Proposition~\ref{Prop:LatticeProperties}(i). Trivially, (d) yields (a).
\end{proof}

\begin{remark} \label{Rem:Nondistributive}
Let $U$ be finite. If $\mathit{RS}(P + Q)$ forms a lattice, it is also a complete lattice equal to $\mathit{IRS}(P + Q)$.
By Proposition~\ref{Prop:Distributivity} we have that $\mathit{RS}(P + Q)$ is distributive if and only if $P$ and $Q$ are coherent.
\end{remark}

\section*{Some concluding remarks}

In this work, we have studied multigranular approximations and rough sets defined by them. The so-called ``optimistic'' 
approximations $\wp(U)_{P + Q}$ and $\wp(U)^{P + Q}$ form dually isomorphic complete lattices which in general are not distributive.
The distributivity is in fact equivalent to the equivalences $P$ and $Q$ being coherent. 
Actually, if $P$ and $Q$ are coherent, then $\wp(U)_{P + Q}$ and $\wp(U)^{P + Q}$ are equal to the
complete atomistic Boolean lattice $\wp(U)_{P \cap Q} = \wp(U)^{P \cap Q}$.
The rough set system $\mathit{RS}(P + Q)$ determined by optimistic approximations is not necessarily a lattice. 
We introduced a condition (C) in terms of singleton equivalence classes of $P$, $Q$ and $P \cap Q$, and showed that whenever (C) holds,
$\mathit{RS}(P + Q)$ is a complete lattice. We also presented the smallest completion $\mathit{IRS}(P + Q)$ of
$\mathit{RS}(P + Q)$ containing it, and showed that $\mathit{IRS}(P + Q)$ is distributive if and only if $P$ and $Q$
are coherent. On the other hand, coherence of $P$ and $Q$ implies that $\mathit{RS}(P + Q)$ is a complete 
regular double Stone lattice. Moreover, if $U$ is finite, then distributivity of $\mathit{RS}(P + Q)$ gives that $P$ and $Q$
are coherent.

Because $P \cup Q$ is a tolerance, the structure of ``pessimistic'' approximations $\wp(U)_{P \cup Q}$ and $\wp(U)^{P \cup Q}$ 
is largely known by existing literature. In general, $\wp(U)_{P \cup Q}$ and $\wp(U)^{P \cup Q}$ are self-dual
mutually isomorphic complete lattices. In addition, if $P \cup Q$ is induced by an irredundant covering, then
$\wp(U)_{P \cup Q}$ and $\wp(U)^{P \cup Q}$ are complete atomistic Boolean lattices. If $P$ and $Q$ are coherent,
then $P \cup Q$ is an equivalence and $\wp(U)_{P \cup Q} = \wp(U)^{P \cup Q}$ is also a complete atomistic Boolean lattice.
The ordered set $\mathit{RS}(P \cup Q)$ is not generally a lattice, but if $P \cup Q$ is induced by an irredundant covering, then
$\mathit{RS}(P \cup Q)$ a complete lattice forming a regular pseudocomplemented Kleene algebra. If $P$ and $Q$ are coherent, then 
$\mathit{RS}(P \cup Q)$ is a complete  regular double Stone lattice.


\end{document}